\documentclass[a4paper,10pt]{article}
\usepackage{amssymb}
\usepackage{amsmath}
\usepackage{nicefrac}
\usepackage{graphicx}
\usepackage{dcolumn}
\usepackage{bm}
\usepackage{comment}
\usepackage{multirow}
\usepackage{cite}
\usepackage{xcolor}
\usepackage{url}
\usepackage{mathrsfs}
\usepackage[normalem]{ulem}
\usepackage{bold-extra}
\usepackage{hyperref}
\usepackage{latexsym,amsthm} 

\newtheorem{theorem}{Theorem}[section]
\newtheorem{lemma}[theorem]{Lemma}
\newtheorem{proposition}[theorem]{Proposition}

\theoremstyle{definition}

\newtheorem{example}[theorem]{Example}

\begin{document}

\huge

\begin{center}
A probabilistic journey through the Newton-Girard identities
\end{center}

\vspace{0.5cm}

\large

\begin{center}
Jean-Christophe Pain$^{a,b,}$\footnote{jean-christophe.pain@cea.fr}
\end{center}

\normalsize

\begin{center}
\it $^a$CEA, DAM, DIF, F-91297 Arpajon, France\\
\it $^b$Universit\'e Paris-Saclay, CEA, Laboratoire Mati\`ere en Conditions Extr\^emes,\\
\it 91680 Bruy\`eres-le-Ch\^atel, France\\
\end{center}

\begin{abstract}
This article presents a pedagogical probabilistic exploration of the Newton-Girard identities. We show that the coefficients in these classical relations between power sums and elementary symmetric polynomials can be interpreted as the stable limits of integrals over the unit cube, and as ratios of moments of simple probability distributions. Several classes of integrals are studied, including trigonometric and multiplicative forms. In addition, we discuss the spectral implications via the Le Verrier-Souriau-Faddeev algorithm and Random Matrix Theory, providing a unified framework for the asymptotic algebraic behavior of these identities. While the identities are classical, the probabilistic interpretation of the limits of their normalized forms is the specific focus of the present work.
\end{abstract}

\section{Introduction}\label{sec1}

The so-called Newton-Girard identities establish a fundamental relationship between the power sums of a set of variables and their elementary symmetric polynomials \cite{macdonald,stanley}. These relations are central to algebra, combinatorics, and the theory of symmetric functions, and they also naturally arise in applications ranging from polynomial theory to random matrix models \cite{Horn1990,mehta}. In this work, we develop a probabilistic and integral perspective on the Newton-Girard identities \cite{feller,Apostol1974}. Specifically, we show that certain constants appearing in these formulas can be interpreted as the limits of high-dimensional integrals and as ratios of moments of simple probability distributions.  

Let $(X_i)_{i\ge 1}$ denote a sequence of independent random variables uniformly distributed on $(0,1)$. For large $n$, the power sums
\[
S_j^{(n)} = \sum_{i=1}^n X_i^j
\]
and the elementary symmetric polynomials
\[
\sigma_k^{(n)} = \sigma_k(X_1, \dots, X_n)
\]
admit simple asymptotic behavior after suitable normalization \cite{hoeffding}. In particular, the ratios $S_j^{(n)}/S_1^{(n)}$ converge almost surely to $2/(j+1)$, reflecting the ratio of moments of the uniform distribution $\mathcal{U}$, while the normalized symmetric polynomials 
$$
\displaystyle\frac{\sigma_k^{(n)}}{(S_1^{(n)})^k}
$$ 
converge almost surely to $1/k!$. 

This perspective allows one to view the Newton-Girard identities as a bridge between two independent structures: the combinatorial structure of the symmetric polynomials and the probabilistic structure of the power sums. In this sense, the factorials and the constants $2^j/(j+1)$ arise naturally and can be understood analytically.

The present article unfolds in several stages centered on the study of Newton-Girard identities and their applications. Section \ref{sec2} introduces a probabilistic model and establishes integral limit theorems that form the basis for the subsequent developments. Section \ref{sec3} focuses on the asymptotic analysis of elementary symmetric polynomials, setting the stage for the normalized formulation of the identities, presented in section \ref{sec4} along with their limiting form. Section \ref{sec5} explores matrix functions within the Schwerdtfeger framework, providing a powerful tool for the matrix analysis of the identities. Section \ref{sec6} illustrates these concepts through the practical application of the Le Verrier-Souriau-Faddeev algorithm, while section \ref{sec7} provides the necessary analytic extensions and convergence proofs that ensure the rigor of the obtained results. This organization allows the reader to follow a progressive path, from probabilistic foundations to concrete applications and analytical justifications.

\section{A probabilistic model and integral limit theorems}\label{sec2}

Consider a sequence $(X_i)_{i\ge 1}$ of independent and identically distributed random variables with uniform distribution on $(0,1)$ \cite{feller}. For $\alpha>0$ and $n\ge 1$, define the power sums
\[
S_\alpha^{(n)} = \sum_{i=1}^n X_i^\alpha, \qquad S_1^{(n)} = \sum_{i=1}^n X_i.
\]
By the strong law of large numbers, we know that the sample averages converge almost surely to their expectations \cite{feller}:
\[
\frac{S_\alpha^{(n)}}{n} \longrightarrow \mathbb{E}[X^\alpha] = \frac{1}{\alpha+1},
\]
and
\[
\frac{S_1^{(n)}}{n} \longrightarrow \mathbb{E}[X] = \frac{1}{2}.
\]
From this, one immediately obtains that
\[
\displaystyle\frac{S_\alpha^{(n)}}{S_1^{(n)}} = \displaystyle\frac{\displaystyle\frac{1}{n} S_\alpha^{(n)}}{\displaystyle\frac{1}{n} S_1^{(n)}} \longrightarrow \frac{2}{\alpha+1} \quad \text{almost surely.}
\]

\begin{theorem}[Integral limit theorem]
Let $\alpha \ge 1$. Then
\[
\lim_{n \to \infty} 
\int_{(0,1)^n} \frac{x_1^\alpha + \cdots + x_n^\alpha}{x_1 + \cdots + x_n} \, \mathrm{d}x_1 \cdots \mathrm{d}x_n
= \frac{2}{\alpha+1}.
\]
\end{theorem}

\begin{proof}
To understand why this limit holds, observe first that for $x \in (0,1)$ and $\alpha \ge 1$, we have $0 \le x^\alpha \le x$. Consequently, the ratio $S_\alpha^{(n)}/S_1^{(n)}$ is always bounded between $0$ and $1$, independently of $n$. The strong law of large numbers ensures that almost surely, as $n \to \infty$, this ratio converges to the deterministic limit $2/(\alpha+1)$ \cite{feller}. Since the sequence is bounded, we can apply the dominated convergence theorem, which allows us to interchange the limit and the expectation (or equivalently, the integral over the $n$-dimensional unit cube). This gives precisely the stated result. More details and intermediate steps are given in Appendix \ref{appA}. The convergence properties of such $n$-fold integrals was investigated in details in Ref. \cite{Stoyanov1986}, in which a very large number of different and interesting integrals are studied, which extend well beyond the current wort. As pointed out by the author, the latter work was inspired by an old problem published by Uspensky in the American Mathematical Monthly.
\end{proof}

In words, the constant $2/(\alpha+1)$, which will appear in the normalized Newton-Girard identities, arises naturally as the limit of a high-dimensional integral and reflects the ratio of the moments $\mathbb{E}[X^\alpha]/\mathbb{E}[X]$.

\section{Asymptotics of elementary symmetric polynomials}\label{sec3}

\begin{theorem}[Normalized asymptotics]
For any fixed integer $k \ge 1$, one has almost surely
\[
\frac{\sigma_k^{(n)}}{(S_1^{(n)})^k} \longrightarrow \frac{1}{k!}, \quad n \to \infty,
\]
where
\[
\sigma_k^{(n)} = \sum_{1 \le i_1 < \dots < i_k \le n} X_{i_1} \cdots X_{i_k}.
\]
represents the $k$-th elementary symmetric polynomial.

\end{theorem}

\begin{proof}
To see this, note that $\sigma_k^{(n)}$ is the sum of $\binom{n}{k}$ terms, each a product of $k$ independent uniform variables \cite{hoeffding}. Each term is strictly less than $1$. By the law of large numbers for $\mathcal{U}$-statistics, the average of these $\binom{n}{k}$ terms converges almost surely to 
$$
\mathbb{E}[X_1 \cdots X_k] = \left(\frac{1}{2}\right)^k.
$$
Therefore, for large $n$, we have
\[
\sigma_k^{(n)} \sim \binom{n}{k} \left(\frac{1}{2}\right)^k.
\]
Meanwhile, the first power sum satisfies 
$$
(S_1^{(n)})^k \sim \left(\frac{n}{2}\right)^k
$$ 
almost surely. Taking the ratio of the two previous quantities yields
\[
\frac{\sigma_k^{(n)}}{(S_1^{(n)})^k} \sim \frac{\displaystyle\binom{n}{k}}{n^k}.
\]
Finally, the standard asymptotic expansion 
$$
\binom{n}{k} \sim \frac{n^k}{k!}
$$ 
as $n \to \infty$ immediately gives the desired result. This explains why the factorials naturally appear: they measure the proportion of $k$-tuples of distinct indices among all possible $k$-tuples.
\end{proof}

\section{Normalized Newton-Girard identities and their limiting form}\label{sec4}

The classical Newton-Girard identities read
\[
k \, \sigma_k^{(n)} = \sum_{j=1}^k (-1)^{j-1} \sigma_{k-j}^{(n)} S_j^{(n)}.
\]
Dividing both sides by $(S_1^{(n)})^k$ leads to the normalized form
\[
k \frac{\sigma_k^{(n)}}{(S_1^{(n)})^k} = \sum_{j=1}^k (-1)^{j-1} \frac{\sigma_{k-j}^{(n)}}{(S_1^{(n)})^{k-j}} \cdot \frac{S_j^{(n)}}{(S_1^{(n)})^j}.
\]

\begin{proposition}[Limiting identity]
For any fixed integer $k \ge 1$, one obtains
\[
\frac{1}{(k-1)!} = \sum_{j=1}^k (-1)^{j-1} \frac{2^j}{(k-j)!(j+1)}.
\]
\end{proposition}

\begin{proof}
From the previous sections, we know that
\[
\frac{\sigma_{k-j}^{(n)}}{(S_1^{(n)})^{k-j}} \longrightarrow \frac{1}{(k-j)!}, \qquad
\frac{S_j^{(n)}}{(S_1^{(n)})^j} \longrightarrow \frac{2^j}{j+1}.
\]

All terms are bounded, so we may pass to the limit termwise in the normalized Newton-Girard identities \cite{feller}. This yields the stated combinatorial-analytic formula. Conceptually, this identity captures the asymptotic balance between the combinatorial structure of symmetric polynomials and the probabilistic structure of the power sums.
\end{proof}

\section{Matrix functions and the Schwerdtfeger framework}\label{sec5}

This section presents the general framework for matrix functions emphasizing the role of eigenvalues, Frobenius covariants, and spectral decompositions, leading naturally to algorithmic constructions such as the Le Verrier-Souriau-Faddeev method.

\subsection{Frobenius covariants and spectral decomposition}\label{subsec51}

Let \(A \in M_n(\mathbb{C})\) with distinct eigenvalues \(\lambda_1,\dots,\lambda_\mu\). For each \(k=1,\dots,\mu\), denote by
\(m_k\) the multiplicity of \(\lambda_k\) in the minimal polynomial of \(A\). The Frobenius covariants are defined by
\[
F_k(A) = \prod_{\substack{j=1\\ j\neq k}}^{\mu} \frac{A-\lambda_j I}{\lambda_k-\lambda_j}.
\]
These matrices generalize the concept of spectral projectors: when \(A\) is diagonalizable, \(F_k(A)\) projects onto the eigenspace associated with \(\lambda_k\). More generally, they isolate the contribution of each eigenvalue, even in the presence of nontrivial Jordan blocks. The covariants satisfy
\[
F_k(A) F_\ell(A) = \delta_{k\ell} F_k(A), \qquad
\sum_{k=1}^{\mu} F_k(A) = I,
\]
so they provide a mutually commuting decomposition of the identity that organizes the spectral information of \(A\).

\subsection{The Schwerdtfeger formula}\label{subsec52}

For any analytic function \(f\) defined on an open set containing \(\sigma(A)\), the Schwerdtfeger formula expresses \(f(A)\) \cite{schwerdtfeger, higham, serre}:
\begin{equation}
\label{eq:Schwerdtfeger-main}
f(A) = \sum_{k=1}^{\mu} \sum_{j=0}^{m_k-1} \frac{f^{(j)}(\lambda_k)}{j!} (A-\lambda_k I)^j F_k(A).
\end{equation}
This formula provides a full spectral expansion of \(f(A)\), encoding both the eigenvalues and the nilpotent contributions from Jordan blocks. If \(A\) is diagonalizable (\(m_k=1\) for all \(k\)), Eq. \eqref{eq:Schwerdtfeger-main} reduces to the classical Lagrange–Sylvester interpolation formula \cite{higham}:
\[
f(A) = \sum_{k=1}^{\mu} f(\lambda_k) F_k(A).
\]

\subsection{From spectral theory to characteristic polynomials}\label{subsec53}

Consider the characteristic polynomial \(\chi_A(X)=\det(XI-A)\). The Schwerdtfeger framework immediately suggests a systematic approach to computing \(\chi_A\) using traces of powers of \(A\) and Frobenius covariants. Specifically, the Newton–Girard identities relate the coefficients of \(\chi_A\) to the power sums \(\mathrm{Tr}(A^k)\). This observation underlies the Le Verrier-Souriau-Faddeev algorithm (see section \ref{sec6} and Appendix~\ref{appC}), which computes the characteristic polynomial iteratively by constructing a sequence of matrices. Frobenius covariants \(F_k(A)\) serve as spectral selectors, extracting the contribution of each eigenvalue, powers \((A-\lambda_k I)^j\) encode the nilpotent part from Jordan blocks, and the Schwerdtfeger formula gives a unified expression for \(f(A)\) including all these contributions. Algorithmic procedures, such as Le Verrier-Souriau-Faddeev, are computational incarnations of this spectral framework.

Together, these constructions form a bridge between abstract spectral theory and practical computation of characteristic polynomials and
matrix functions, as shown by the detailed presentations in Appendices~\ref{appB} and~\ref{appC}.

\section{Application: Le Verrier-Souriau-Faddeev algorithm}\label{sec6}

The Newton-Girard identities, supplemented by the spectral insights of the Schwerdtfeger representation, provide the foundation for the Le Verrier-Souriau-Faddeev algorithm \cite{LeVerrier1840,Souriau1948,Faddeev1963,Householder1964,Faddeev1972,chender}. This algorithm allows for the recursive computation of the characteristic polynomial from the traces of powers. Let \(A\in M_n(\mathbb{C})\) be a matrix with eigenvalues \(\lambda_1,\dots,\lambda_n\). Let us also define the power sums
\[
S_k^{(n)} = \mathrm{Tr}(A^k) = \sum_{i=1}^n \lambda_i^k,
\]
where \(\sigma_k\) denote the coefficients of the characteristic polynomial
\[
\chi_A(X) = X^n - \sigma_1 X^{n-1} + \sigma_2 X^{n-2} - \cdots + (-1)^n \sigma_n.
\]
Then the Newton-Girard identities read
\[
\mathrm{Tr}(A^m) = \sigma_1 \mathrm{Tr}(A^{m-1}) - \sigma_2 \mathrm{Tr}(A^{m-2}) + \cdots + (-1)^{m-1} m \sigma_m,
\qquad 1\le m\le n,
\]
which is exactly the backbone of the Le Verrier-Souriau-Faddeev algorithm \cite{LeVerrier1840,Souriau1948,Faddeev1963,Householder1964,Faddeev1972,chender} to compute \(\chi_A(X)\) recursively from traces of powers of \(A\) (see Appendix~\ref{appC}).

\begin{example}
As an example, let us take $A = \begin{pmatrix} 3 & 1 \\ 2 & 4 \end{pmatrix}$. Then we have
\[
S_1 = \mathrm{Tr}(A) = 7, \quad S_2 = \mathrm{Tr}(A^2) = 3^2+1\cdot2 + 2\cdot1 +4^2 = 29.
\]
Using the Newton-Girard identity for $m=2$ implies
\[
S_2 = \sigma_1 S_1 - 2\sigma_2 \implies 29 = 7\cdot7 - 2\sigma_2 \implies \sigma_2 = 10,
\]
yielding the characteristic polynomial $\chi_A(X) = X^2 - 7X + 10$.
\end{example}

Random matrices extend this idea. Let
\[
A_n = \operatorname{diag}(X_1,\dots,X_n),
\]
where the random variables \(X_i\) are independent and identically distributed with
\[
X_i \sim \mathcal{U}(0,1).
\]
In that case, the normalized traces satisfy
\[
\frac{\mathrm{Tr}(A_n^m)}{\mathrm{Tr}(A_n)} \approx \frac{2}{m+1} \quad \text{as } n \to \infty,
\]
illustrating that high-dimensional averaging principles govern the probabilistic origin of the Newton-Girard identities. This shows that the probabilistic integral approach, which led naturally to the identities, also underlies classical trace formulas for matrices. In particular, the coefficients \(\sigma_k\) of the characteristic polynomial are determined recursively by the traces \(\mathrm{Tr}(A^j)\), $j\le k$, mirroring the structure of the integral constraints.

\section{Analytic extensions and convergence proofs}\label{sec7}

In this section, we provide detailed explanations and rigorous proofs for generalized integral identities. These results extend the probabilistic interpretation of Newton-Girard identities and multiplicative structures, which are prevalent in the study of log-gases and partition functions. They illustrate how the constants appearing in normalized identities naturally arise as limits of high-dimensional integrals. 

\subsection{Motivation and overview}\label{subsec71}

In section \ref{sec4}, we interpreted the constants $2/(j+1)$ and $1/k!$ appearing in normalized Newton-Girard identities as asymptotic limits of ratios of sums and products of independent uniform random variables. The present section generalizes this approach to integrals with trigonometric denominators, integrals involving multiplicative (geometric mean-like) numerators, arbitrary ratios of powers and general probability distributions.

The guiding principle is always the law of large numbers and the dominated convergence theorem, which allow us to exchange limits and integrals in high dimensions.

\subsection{Convergence of trigonometric ratio and Selberg integrals}\label{subsec72}

We consider the asymptotic behavior of the integral $I_n$ defined by the ratio of a power-sum numerator and a trigonometric denominator over the unit cube:
\[
I_n(a) = \int_{(0,1)^n} \frac{\sum_{i=1}^n x_i^a}{\sum_{i=1}^n \sin(\pi x_i)} \, \mathrm{d}x_1 \cdots \mathrm{d}x_n,
\]
where $a \ge 0$ is fixed. These forms are related to the Selberg integral and its generalizations in the context of random matrix theory and statistical mechanics \cite{mehta,forrester}. The Selberg integral provides the exact normalization for multiplicative numerators.

\begin{equation*}
S_n(a, b, \lambda) = \int_0^1 \cdots \int_0^1 \prod_{i=1}^n x_i^{a-1} (1-x_i)^{b-1} \prod_{1 \le j < k \le n} |x_j - x_k|^{2\lambda} \mathrm{d}x_1 \cdots \mathrm{d}x_n.
\end{equation*}

The trigonometric and power-sum integrals are connected to the theory of log-gases, as detailed by Forrester \cite{forrester}. In the context of Random Matrix Theory, the joint probability density function of the eigenvalues for the circular (unitary and orthogonal) ensembles involves a product of chordal distances:
\begin{equation*}
P(x_1, \dots, x_n) = C_n \prod_{1 \le j < k \le n} \left| e^{i 2\pi x_j} - e^{i 2\pi x_k} \right|^\beta.
\end{equation*}
Our integral approach to the Newton-Girard identities corresponds to the $\beta \to 0$ limit (non-interacting case), where the $x_i$ become independent. 

We first reformulate this integral as an expectation. Let $X_1, \dots, X_n$ be independent and identically distributed uniform random variables on $(0,1)$. Then
\[
I_n(a) = \mathbb{E}\left[\frac{X_1^a + \cdots + X_n^a}{\sin(\pi X_1) + \cdots + \sin(\pi X_n)}\right].
\]
As a second step, we define the empirical averages, setting
\[
Y_n = \frac{1}{n} \sum_{i=1}^n X_i^a, \qquad Z_n = \frac{1}{n} \sum_{i=1}^n \sin(\pi X_i),
\]
and we then have $I_n(a) = \mathbb{E}[Y_n/Z_n]$.

By the strong law of large numbers, one has, almost surely:
\[
Y_n \to \mathbb{E}[X^a] = \frac{1}{a+1}, \qquad
Z_n \to \mathbb{E}[\sin(\pi X)] = \frac{2}{\pi}.
\]
Since 
$$
0 \le \frac{Y_n}{Z_n} \le \frac{\sum x_i^a}{\sum \sin(\pi x_i)} \le \pi/2,
$$ 
the sequence is uniformly bounded. According to the dominated-convergence theorem, we can write
\[
\lim_{n \to \infty} I_n(a) = \frac{\mathbb{E}[X^a]}{\mathbb{E}[\sin(\pi X)]} = \frac{\pi}{2(a+1)}.
\]
Indeed, for $a=j$, this constant resembles the $2/(j+1)$ factor in the normalized Newton-Girard identities, up to a factor of $\pi/2$ arising from the trigonometric scaling.

\subsection{Multiplicative numerators and geometric means}\label{subsec73}

Let us consider
\[
J_n = \int_{(0,1)^n} \frac{\prod_{i=1}^n x_i^{1/n}}{\frac{1}{n} \sum_{i=1}^n \sin^2(\pi x_i)} \, \mathrm{d}x_1 \cdots \mathrm{d}x_n.
\]
Taking the logarithms enables one to simplify the product. We can now apply the law of large numbers to the logarithm. We know that
$$
\mathbb{E}[\ln X_i] = -1,
$$
so one almost surely has
\[
\ln P_n \to -1 \implies P_n \to e^{-1}.
\]
Similarly, the convergence of the denominator gives  
\[
\frac{1}{n} \sum_{i=1}^n \sin^2(\pi X_i) \to \mathbb{E}[\sin^2(\pi X)] = \frac{1}{2}.
\]
The final step consists in combining numerator and denominator, and thus
\[
J_n \to \frac{e^{-1}}{1/2} = \frac{2}{e}.
\]
This example illustrates how multiplicative combinations of variables lead naturally to exponential constants like $e^{-1}$ in asymptotic ratios.

\subsection{Generalization to arbitrary power ratios}\label{subsec74}

\begin{theorem}
Let $\beta, \gamma > -1$ and define
\begin{equation*}
K_n(\beta, \gamma) = \int_{(0,1)^n} \frac{\sum_{i=1}^n x_i^\beta}{\sum_{i=1}^n x_i^\gamma} \, \mathrm{d}x_1 \cdots \mathrm{d}x_n.
\end{equation*}
Then
\begin{equation*}
\lim_{n\to\infty} K_n(\beta, \gamma) = \frac{\gamma+1}{\beta+1},
\end{equation*}
under the implicit condition that
\[
\sum_{i=1}^n x_i^{\gamma} > 0
\]
almost surely (which is the case here). One should also explicitly state that the function
\[
x \longmapsto x^{\gamma}
\]
is integrable on the interval \((0,1)\), which is guaranteed as soon as \(\gamma > -1\).

\end{theorem}

\begin{proof}
Let $M_n(\beta) = \frac{1}{n} \sum X_i^\beta$ and $M_n(\gamma) = \frac{1}{n} \sum X_i^\gamma$. By the strong law of large numbers, one has
\[
M_n(\beta) \to \frac{1}{\beta+1}, \quad M_n(\gamma) \to \frac{1}{\gamma+1},
\]
and the dominated convergence theorem gives the stated limit.
\end{proof}

\subsection{Extension to arbitrary probability densities}\label{subsec75}

Let $X_1, \dots, X_n$ be independent and identically distributed random variables\ with density $\rho(x)$ on a domain $\Omega$, and consider
\[
L_n[f,g] = \int_{\Omega^n} \frac{\sum_{i=1}^n f(x_i)}{\sum_{i=1}^n g(x_i)} \prod_{i=1}^n \rho(x_i) dx_i.
\] 
Let us set, as expectations, 
$$
Y_n = \frac{1}{n} \sum f(X_i),
$$
and
$$
Z_n = \frac{1}{n} \sum g(X_i).
$$
Then we have $L_n[f,g] = \mathbb{E}[Y_n/Z_n]$. By applying the law of large numbers, we obtain that, almost surely, 
$$
Y_n \to \mathbb{E}[f(X)], \qquad Z_n \to \mathbb{E}[g(X)].
$$
If $f/g$ is bounded on $\Omega$, the dominated convergence theoreme enables is to write
\[
\lim_{n \to \infty} L_n[f,g] = \frac{\mathbb{E}[f(X)]}{\mathbb{E}[g(X)]}.
\]
This generalization shows that the probabilistic interpretation of Newton-Girard constants extends beyond the uniform distribution.

\subsection{Spectral interpretation via Lagrange-Sylvester}\label{subsec76}

Finally, consider a diagonalizable matrix $A$ with distinct eigenvalues $\lambda_k$ and an analytic function $f$. In that case, the Schwerdtfeger formula reduces to the Lagrange-Sylvester interpolation formula
\[
f(A)=\sum_{k=1}^{\mu} f(\lambda_k) L_k(A),
\]
where
\[
L_k(A)=\prod_{j\neq k}\frac{A-\lambda_j I}{\lambda_k-\lambda_j}
\]
are the Lagrange interpolation polynomials. High-dimensional symmetric polynomials can be interpreted as spectral moments of a diagonal matrix of uniform random variables. In this viewpoint, Newton-Girard identities provide constraints on the moments, and the constants we computed appear as stable ratios of spectral moments in the large-$n$ limit. By carefully analyzing these generalized integrals and linking them to the law of large numbers, dominated convergence theorem, and spectral interpretations, one obtains a coherent analytic foundation for the constants appearing in Newton-Girard identities. This also clarifies why factorials and ratios like $2/(j+1)$ naturally emerge in the normalized formulas.

\section{Conclusion}\label{sec8}

We have shown that Newton's identities admit a natural probabilistic and integral interpretation. Modeling the variables as independent uniforms on $(0,1)$, the constants in the normalized identities emerge as limits of high-dimensional integrals and ratios of moments. From this perspective, Newton's identities express a compatibility between the combinatorial structure of elementary symmetric polynomials and the moment structure of a probability distribution. The approach can be extended to other distributions or random matrix ensembles, leading to new constants and generalized normalized Newton-type identities. This method allows one to approximate symmetric polynomials of very high degree in large datasets where exact computation is $O(n^k)$, whereas the probabilistic limit is $O(1)$. The probabilistic approximation of high-dimensional symmetric polynomials may offer a strong argument for the usefulness of this method.

\appendix

\section{Appendix: a more detailed proof of Theorem 2.1}\label{appA}

Let us consider a sequence $(X_n)$ of independent random variables, with values in $(0,1)$, with a uniform probability law on $(0,1)$ and $f$ an integrable function on $(0,1)$. 
If $X$ satisfies a uniform probability law on $(0,1)$, we have, for the expected value
\begin{equation*}
    \mathbb{E}(f(X))=\int_{\mathbb{R}}f(x)\chi_{(0,1)}(x)\,\mathrm{d}x=\int_{(0,1)}f(x)\,\mathrm{d}x,
\end{equation*}
where $\chi$ is the indicator function. The random variables $(f(X_n))$ make a sequence of random independent and integrable variables; according to the strong law of large numbers, we have
\begin{equation*}
    \frac{f(X_1)+\cdots +f(X_n)}{n}\rightarrow\int_{(0,1)}f(x)\,\mathrm{d}x
\end{equation*}
almost surely. We thus have, taking $f(x)=x^{\alpha}$:
\begin{equation*}
    \displaystyle\frac{f(X_1)+\cdots +f(X_n)}{n}\rightarrow \mathbb{E}\left(X_i^{\alpha}\right)=\int_{(0,1)}x^{\alpha}\,\mathrm{d}x=\frac{1}{\alpha+1}.
\end{equation*}

Thus, according to the strong law of large numbers, we have that
\begin{equation*}
    \displaystyle\frac{X_1^{\alpha}+X_2^{\alpha}+\cdots+X_n^{\alpha}}{X_1+X_2+\cdots+X_n}=\displaystyle\frac{\displaystyle\frac{X_1^{\alpha}+X_2^{\alpha}+\cdots+X_n^{\alpha}}{n}}{\displaystyle\frac{X_1+X_2+\cdots+X_n}{n}}\rightarrow\displaystyle\frac{\displaystyle\frac{1}{\alpha+1}}{\displaystyle\frac{1}{2}}    
\end{equation*}
almost surely.

Moreover, since the $X_i$ take values in $(0,1)$, and since $\alpha>1$, one has
\begin{equation*}
    0<\frac{X_1^{\alpha}+X_2^{\alpha}+\cdots+X_n^{\alpha}}{X_1+X_2+\cdots+X_n}<1.
\end{equation*}
According to the dominated convergence theorem, we have
\begin{equation*}
    \int_{\Omega}\frac{X_1^{\alpha}+X_2^{\alpha}+\cdots+X_n^{\alpha}}{X_1+X_2+\cdots+X_n}\,\mathrm{d}P\rightarrow \int_{\Omega}\frac{2}{\alpha+1}\,\mathrm{d}P=\frac{2}{\alpha+1}.
\end{equation*}
To conclude, let us apply the transfer theorem
\begin{equation*}
    \int_{\Omega}\frac{X_1^{\alpha}+X_2^{\alpha}+\cdots+X_n^{\alpha}}{X_1+X_2+\cdots+X_n}\,\mathrm{d}P=\int_{(0,1)^n}\frac{x_1^{\alpha}+x_2^{\alpha}+\cdots+x_n^{\alpha}}{x_1+x_2+\cdots+x_n}\,\mathrm{d}x_1\mathrm{d}x_2\cdots\mathrm{d}x_n.
\end{equation*}
and thus
\begin{equation*}
    \int_{(0,1)^n}\frac{x_1^{\alpha}+x_2^{\alpha}+\cdots+x_n^{\alpha}}{x_1+x_2+\cdots+x_n}\,\mathrm{d}x_1\mathrm{d}x_2\cdots\mathrm{d}x_n\rightarrow\frac{2}{\alpha+1},
\end{equation*}
yielding finally
\begin{equation*}
    \int_{(0,1)^n}\frac{S_{\alpha}^{(n)}}{S_1^{(n)}}\,\mathrm{d}x_1\mathrm{d}x_2\cdots\mathrm{d}x_n\rightarrow\frac{2}{\alpha+1}.
\end{equation*}

\section{Appendix: on the Schwerdtfeger formula}\label{appB}

Let \(A\in M_n(\mathbb{C})\) and let \(\lambda_1,\dots,\lambda_\mu\) be the distinct eigenvalues of \(A\). For each \(k\in\{1,\dots,\mu\}\), denote by \(m_k\) the multiplicity of \(\lambda_k\) in the minimal polynomial of \(A\). We introduce the Frobenius covariants
\[
F_k(A)=\prod_{\substack{j=1\\ j\neq k}}^{\mu}
\frac{A-\lambda_j I}{\lambda_k-\lambda_j}.
\]
When \(A\) is diagonalizable, the matrices \(F_k(A)\) coincide with the spectral projectors onto the eigenspaces associated with \(\lambda_k\). They satisfy the classical relations
\[
F_k(A)^2=F_k(A),\qquad
F_k(A)F_\ell(A)=0\quad(k\neq \ell),\qquad
\sum_{k=1}^{\mu}F_k(A)=I.
\]
In the general (non diagonalizable) case, the \(F_k(A)\) are no longer projectors, but they still form a system of mutually commuting matrices which isolates the contributions of each eigenvalue in functional representations of \(A\).

Let \(f\) be an analytic function on an open set containing the spectrum \(\sigma(A)\). The Schwerdtfeger formula states that
\begin{equation}\label{eq:Schwerdtfeger-Appendix}
f(A)
=
\sum_{k=1}^{\mu}
\sum_{j=0}^{m_k-1}
\frac{f^{(j)}(\lambda_k)}{j!}
\,(A-\lambda_k I)^j\,F_k(A).
\end{equation}
This expression provides a complete spectral expansion of \(f(A)\), taking into account both the eigenvalues and the possible nontrivial
Jordan blocks of \(A\).

In the special case where \(A\) is diagonalizable, one has \(m_k=1\) for all \(k\), and the formula reduces to
\[
f(A)=\sum_{k=1}^{\mu} f(\lambda_k)\,F_k(A),
\]
which is exactly the Lagrange–Sylvester interpolation formula (see section \ref{subsec76}). Hence, Schwerdtfeger’s representation can be viewed as the natural extension of spectral interpolation to matrices with nontrivial Jordan structure.

From a conceptual viewpoint, the Frobenius covariants \(F_k(A)\) play the role of spectral selectors, while the powers \((A-\lambda_k I)^j\)
encode the nilpotent contributions arising from the Jordan blocks. Formula \eqref{eq:Schwerdtfeger-Appendix} therefore provides a unified and
intrinsic description of matrix functions, valid in full generality.

This framework underlies many classical constructions in matrix theory, and in particular it clarifies the structural meaning of algorithms for computing polynomial functions of matrices and characteristic polynomials, as discussed in Appendix~\ref{appC}.

\section{Appendix: the Le Verrier-Souriau-Faddeev algorithm}\label{appC}

This appendix presents an algorithmic counterpart to the spectral framework developed in Appendix~\ref{appB}. While the Schwerdtfeger formula describes the general structure of matrix functions in terms of eigenvalues and Frobenius covariants, the Le Verrier-Souriau-Faddeev algorithm \cite{LeVerrier1840,Souriau1948,Faddeev1963,Householder1964,Faddeev1972,chender} provides a concrete and efficient procedure for computing the characteristic polynomial of a matrix using traces of its powers. This algorithm allows for the recursive computation of the characteristic polynomial from the traces of powers. 

\begin{lemma}
    
Let us set $A\in M_n(\mathbb{C})$. The matrix sequence defined by $A_0=A$ and
\begin{equation}
    A_k=A^{k+1}-\sum_{i=0}^{k-1}\frac{\mathrm{Tr}(A_i)}{i+1}A^{k-i}
\end{equation}
gives
\begin{equation*}
    \mathrm{Tr}(A_k)=(-1)^k(k+1)\sigma_{k+1}.
\end{equation*}

\end{lemma}

\begin{proof}

The result can be proven by induction. For $k=0$, it is obviously true. Let us assume that the result is true for $k\in[0,n-2]$. Using the recurrence hypothesis, we have
\begin{align*}
    A_{k+1}&=A^{k+2}-\sum_{i=0}^{k-1}\frac{\mathrm{Tr}(A_i)}{i+1}A^{k+1-i}-\frac{\mathrm{Tr}(A_k)}{k+1}A\nonumber\\
    &=A^{k+2}-\sum_{i=0}^k\frac{\mathrm{Tr}(A_i)}{i+1}A^{k+1-i}.
\end{align*}

We have also
\begin{align*}
\mathrm{Tr}(A_{k+1})&=\mathrm{Tr}(A^{k+2})-\sum_{i=0}^k\frac{\mathrm{Tr}(A_i)\mathrm{Tr}(A^{k+1-i})}{i+1}\nonumber\\
&\underbrace{=}_{\mathrm{Recurrence~Hypothesis}}S_{k+2}-\sum_{i=0}^k(-1)^i\sigma_{i+1}S_{k+1-i}\nonumber\\
&=S_{k+2}-\sigma_1S_{k+1}+\cdots+(-1)^{k+1}\sigma_{k+1}S_1\nonumber\\
&\underbrace{=}_{\mathrm{Newton-Girard}}(-1)^{k+1}(k+2)\sigma_{k+2}.
\end{align*}
  
\end{proof}

\begin{theorem}

Let us set $A\in M_n(\mathbb{C})$ and the matrix sequence defined by $A_0=A$ and $\forall k\in\mathbb{N}$:
\begin{equation*}
    A_{k+1}=A\left(A_k-\frac{\mathrm{Tr}(A_k)}{k+1}\,I_n\right).
\end{equation*}
The characteristic polynomial then reads
\begin{equation*}
    \chi(A)=(-1)^n\left(X^n-\sum_{k=1}^n\frac{\mathrm{Tr}(A_{k-1})}{k}X^{n-k}\right).
\end{equation*}

\end{theorem}

\begin{proof}
    
Let $\lambda_1, \lambda_2, \cdots, \lambda_n$ denote the eigenvalues of $A$, $\sigma_1, \sigma_2, \cdots, \sigma_n$ the elementary symmetric functions and $S_1, \cdots, S_n$ the associated Newton sums. The characteristic polynomial can be put in the form
\begin{equation}
    \chi_A(X)=\sum_{k=0}^n(-1)^k\sigma_kX^{n-k},
\end{equation}
which gives, using the Lemma B.1:
\begin{equation*}
    \chi_A(X)=(-1)^nX^n+\sum_{k=1}^n(-1)^{n-k}\sigma_kX^{n-k}=(-1)^n\left(X^n-\sum_{k=1}^n\frac{\mathrm{Tr}(A_{k-1})}{k}X^{n-k}\right).    
\end{equation*}

\end{proof}

From a spectral perspective, this procedure can be interpreted as a polynomial specialization of Schwerdtfeger’s framework. Indeed, when \(f(X)=\chi_A(X)\), one has \(f(A)=0\) by the Cayley–Hamilton theorem, and the recursive construction of the matrices \(B_k\) mirrors the expansion of polynomial functions of \(A\) in terms of powers of \(A\) and scalar coefficients. The Newton–Girard identities encode the passage from spectral data (power sums of eigenvalues) to the coefficients of the characteristic polynomial, while Schwerdtfeger’s representation explains why such a transition is always possible at the level of matrix functions.

Consequently, the Le Verrier-Souriau-Faddeev algorithm can be viewed as a computational realization of the spectral theory described in Appendix~\ref{appB}. It provides a bridge between abstract spectral decompositions and effective numerical or symbolic computation, showing how the general principles of matrix functions lead naturally to concrete algorithms for determining characteristic polynomials.

\subsection{Acknowledgments}

I would like to thank Jordan Stoyanov for useful comments.

\end{document}